\documentclass[reqno]{amsart}
\usepackage{amssymb, amsthm, amsmath, color}

\theoremstyle{definition}
\newtheorem{theorem}{Theorem}[section]
\newtheorem{definition}[theorem]{Definition}
\newtheorem{proposition}[theorem]{Proposition}
\newtheorem{lemma}[theorem]{Lemma}
\newtheorem{corollary}[theorem]{Corollary}
\newtheorem{example}[theorem]{Example}

\newcommand{\B}{\mathfrak B}

\newcommand{\OL}{{\mathfrak O}_{L}}
\newcommand{\A}{{\mathfrak A}}

\newcommand{\barphi}{\overline{\varphi}}
\renewcommand{\bar}{\overline}

\DeclareMathOperator{\Perm}{\mbox{Perm}}
\DeclareMathOperator{\Hol}{\mbox{Hol}}
\DeclareMathOperator{\Aut}{\mbox{Aut}}

\begin{document}

\title[$ \rho $-conjugate Hopf-Galois structures]{On $ \rho $-conjugate Hopf-Galois structures}

\author{Paul J. Truman}

\address{School of Computing and Mathematics \\ Keele University \\ Staffordshire \\ ST5 5BG \\ UK}

\email{P.J.Truman@Keele.ac.uk}



\begin{abstract}
The Hopf-Galois structures admitted by a given Galois extension of fields $ L/K $ with Galois group $ G $ correspond bijectively with certain subgroups of $ \Perm(G) $. We use a natural partition of the set of such subgroups to obtain a method for partitioning the  set of corresponding Hopf-Galois structures, which we term \textit{$ \rho $-conjugation}. We study properties of this construction, with particular emphasis on the Hopf-Galois analogue of the Galois correspondence, the connection with skew left braces, and applications to questions of integral module structure in extensions of local or global fields. In particular, we show that the number of distinct $ \rho $-conjugates of a given Hopf-Galois structure is determined by the corresponding skew left brace, and that if $ H, H' $ are Hopf algebras giving $ \rho $-conjugate Hopf-Galois structures on a Galois extension of local or global fields $ L/K $ then an ambiguous ideal $ \B $ of $ L $ is free over its associated order in $ H $ if and only if it is free over its associated order in $ H' $. We exhibit a variety of examples arising from interactions with existing constructions in the literature. 
\end{abstract}

\maketitle

\section{Introduction} \label{sec_introduction}

A \textit{Hopf-Galois structure} on a finite extension of fields $ L/K $ consists of a cocommutative $ K $-Hopf algebra $ H $ making $ L $ into an $ H $-module algebra satisfying a certain non-degeneracy condition (see \cite[(2.7) Definition]{TWE} for the full definition). The motivating example is a Galois extension $ L/K $ with Galois group $ G $: in this case the group algebra $ K[G] $ (with its usual action on $ L $) gives a Hopf-Galois structure on the extension. Hopf-Galois structures allow many results from Galois theory to be generalized to extensions which are inseparable or non-normal, but there is also considerable interest in the applications to Galois extensions: various authors have studied the number of Hopf-Galois structures admitted by particular extensions, the structure of the Hopf algebras involved, and questions concerning the Hopf-Galois analogue of the Galois correspondence. We highlight two applications in particular. Firstly, it has recently been discovered that each Hopf-Galois structure on a Galois extension yields a \textit{skew left brace}; these objects are intensively studied because they can be used to generate set-theoretic solutions of the Yang-Baxter equation. The translation of existence, classification, and structural results between Hopf-Galois theory and skew brace theory has enriched both disciplines. Secondly, Hopf-Galois theory has been applied fruitfully to questions originating in Galois module theory: if $ L/K $ is a finite extension of local or global fields and $ H $ is a Hopf algebra giving a Hopf-Galois structure on the extension then we may study each fractional ideal $ \B $ of $ L $ as a module over its \textit{associated order} in $ H $. This raises the possibility of comparing the structure of $ \B $ as a module over its associated orders in the various Hopf-Galois structures admitted by the extension. We refer the reader to \cite{HAGMT} for a recent survey of these and other related topics. 

A theorem of Greither and Pareigis \cite{GP87} classifies the Hopf-Galois structures admitted by a finite separable extension of fields 
$ L/K $. In the special case of  Galois extension with Galois group $ G $ their theorem implies that Hopf-Galois structures admitted by $ L/K $ correspond bijectively with regular subgroups $ N $ of $ \Perm(G) $ that are normalized by the image of the left regular representation $ \lambda : G \hookrightarrow \Perm(G) $. In \cite{KKTU19a} it is noted that such a subgroup $ N $ need not be normalized by the image of \textit{right} regular representation $ \rho : G \hookrightarrow \Perm(G) $, and that for each $ g \in G $ the subgroup $ \rho(g)N\rho(g)^{-1} $ is again regular and normalized by $ \lambda(G) $, and therefore corresponds to a Hopf-Galois structure on $ L/K $. Following \cite{KT22} we call this method of partitioning the set of  Hopf-Galois structures on $ L/K $ \textit{$ \rho $-conjugation}. We review existing results concerning $ \rho $-conjugate Hopf-Galois structures in  more detail in Section \ref{sec_rho_conjugate} of this paper. 

In Section \ref{sec_HG_correspondence} we study the Hopf-Galois analogue of the Galois correspondence, showing that the lattices of subfields realized by $ \rho $-conjugate Hopf-Galois structures are intimately related to one another. Section \ref{sec_skew_braces} concerns the skew braces corresponding to $ \rho $-conjugate Hopf-Galois structures; we show that the number of distinct $ \rho $-conjugates of a given Hopf-Galois structure is determined by the corresponding skew brace, and obtain a criterion for a Hopf-Galois structure to have no nontrivial $ \rho $-conjugates. 

In Section \ref{sec_intergal_module_structure} we suppose that $ L/K $ is a separable extension of local or global fields and study questions of integral module structure. We show that if an ambiguous ideal $ \B $ of $ L $ is free over its associated order in a particular Hopf-Galois structure then it is free over its associated order in each of the Hopf-Galois structures that are $ \rho $-conjugate to the original one. Finally, in Section \ref{sec_interactions} we study how $ \rho $-conjugation interacts with other constructions in the literature, including Hopf-Galois structures arising from fixed point free pairs of homomorphisms (due to Byott and Childs), those arising from abelian maps (instigated by Childs and generalized by Koch), and induced Hopf-Galois structures (due to Crespo, Rio, and Vela).


\section{$ \rho $-conjugate Hopf-Galois structures} \label{sec_rho_conjugate}

In this section we recall the definition of $ \rho $-conjugate Hopf-Galois structures and some of their known properties. 

Let $ L/K $ be a Galois extension of fields with Galois group $ G $. As mentioned in Section \ref{sec_introduction}, a theorem of Griether and Pareigis \cite{GP87} implies that the Hopf-Galois structures admitted by $ L/K $ correspond bijectively with regular subgroups $ N $ of $ \Perm(G) $ that are normalized by the image of $ G $ under the left regular representation $ \lambda : G \hookrightarrow \Perm(G) $.  The normalization condition is equivalent to requiring that $ N $ is stable under the action of $ G $ on $ \Perm(G) $ given by $ g \ast \eta = \lambda(g) \eta \lambda(g)^{-1} $ for all $ \eta \in \Perm(G) $ and all $ g \in G $; accordingly, we shall refer to the subgroups of interest as \textit{$ G $-stable regular subgroups}. 

In general, a $ G $-stable regular subgroup $ N $ of $ \Perm(G) $ need not be normalized by the image of the right regular representation $ \rho : G \hookrightarrow \Perm(G) $. This motivates the following (see also \cite[Example 2.7]{KKTU19a}). 

\begin{proposition} \label{prop_rho_conjugate_G_stable_regular}
Let $ N $ be a $ G $-stable regular subgroup of $ \Perm(G) $, and let $ g \in G $. Then $ N_{g} := \rho(g)N\rho(g)^{-1} $ is a $ G $-stable regular subgroup of $ \Perm(G) $. 
\end{proposition}
\begin{proof}
Clearly $ N_{g} $ is a subgroup of $ \Perm(G) $ of the same order as $ N $. Since $ N $ acts transitively on $ G $, we have
\[N_{g} [g] =  (N [e_{G}])g^{-1} = G g^{-1} = G, \]
so $ N_{g} $ acts transitively on $ G $. Therefore $ N_{g} $ is a regular subgroup of $ \Perm(G) $. To show $ G $-stability, we note that $ \lambda(G) $ and $ \rho(G) $ centralize one another inside $ \Perm(G) $; hence for $ h \in G $ and $ \eta \in N $ we have
\begin{eqnarray*}
h \ast \rho(g) \eta \rho(g)^{-1} & = & \lambda(h) \rho(g) \eta \rho(g)^{-1} \lambda(h)^{-1} \\
& = & \rho(g) \lambda(h)  \eta  \lambda(h)^{-1}\rho(g)^{-1} \\
& = & \rho(g) (h \ast \eta) \rho(g)^{-1}.
\end{eqnarray*}
We have $ h \ast \eta \in N $ because $ N $ is $ G $-stable; hence $ \rho(g) (h \ast \eta) \rho(g)^{-1} \in N_{g} $, and so $ N_{g} $ is $ G $-stable. 
\end{proof}

\begin{definition}
We shall call two $ G $-stable regular subgroups $ N,N' $ of $ \Perm(G) $ \textit{$\rho $-conjugate} to mean that $ N' = N_{g} $ for some $ g \in G $. In this case we shall also say that the Hopf-Galois structures corresponding to $ N, N' $ are $ \rho $-conjugate to one another. 
\end{definition}

It is clear that $ \rho $-conjugation is an equivalence relation on the set of $ G $-stable regular subgroups of $ \Perm(G) $, and so yields an equivalence relation on the set of Hopf-Galois structures admitted by $ L/K $. Given a $ G $-stable regular subgroup $ N $, we have $ N_{g} = N_{h} $ if and only if $ \rho(g^{-1}h) \in \mathrm{Norm}_{\scriptsize{\Perm(G)}}(N) $. In particular, if $ G $ is abelian then $ \rho(G) = \lambda(G) $, and since $ N $ is $ G $-stable we have $ N_{g}=N $ for all $ g \in G $. We shall return to the question of determining the number of distinct  $ \rho $-conjugates of a given $ G $-stable regular subgroup in Section \ref{sec_skew_braces}.  

\begin{example} \label{example_pq}
Let $ p,q $ be prime numbers with $ p \equiv 1 \pmod{q} $ and let $ L/K $ be a Galois extension whose Galois group $ G $ is isomorphic to the metacyclic group of order $ pq $:
\[ G = \langle s, t \mid s^{p}=t^{q}=e, \; t s t^{-1} = s^{d} \rangle, \]
where $ d $ is a positive integer of multiplicative order $ q $ modulo $ p $. Let $ \eta = \lambda(s) \rho(t) \in \Perm(G) $ and $ N = \langle \eta \rangle $. It is routine to verify that $ N $ is $ G $-stable and regular. We find that $ N_{t} = N $ and that
\[ N_{s^{i}} = \langle \lambda(s) \rho(s^{i(1-d)}t) \rangle \]
for each $ i=0, \ldots ,p-1 $. These subgroups are all distinct, and so we obtain a family of $ p $ mutually $ \rho $-conjugate Hopf-Galois structures on $ L/K $. By \cite[Theorem 6.2]{By04c}, these are all the Hopf-Galois structures admitted by $ L/K $ for which the corresponding subgroup of $ \Perm(G) $ is cyclic (that is: those of cyclic \textit{type}). 
\end{example}

\begin{example} \label{example_dihedral}
Let $ n $ be an even natural number and let $ L/K $ be a Galois extension whose Galois group $ G $ is isomorphic to the dihedral group of order $ 2n $:
\[ G = \langle r,s \mid r^{n}=s^{2}=e, \; srs^{-1} = r^{-1} \rangle. \]
Let $ N = \langle \lambda(r) \rho(s), \lambda(s) \rangle $. Then $ N $ is $ G $ stable, regular, and isomorphic to $ D_{2n} $. We find that $ N_{s} = N $ and that
\[ N_{r^{k}} = \langle \lambda(r) \rho(r^{2k}s), \lambda(s) \rangle \] 
for each $ k=0, \ldots ,n/2-1 $. These subgroups are all distinct, and so we obtain a family of $ n/2 $ mutually $ \rho $-conjugate Hopf-Galois structures on $ L/K $. 
\end{example}

We shall shed more light on these examples, and construct others, in Section \ref{sec_interactions}. 

In Greither and Pareigis's original paper classifying Hopf-Galois structures on separable extensions \cite{GP87} they show that if $ N $ is a $ G $-stable regular subgroup of $ \Perm(G) $ then so too is 
\[ N^{opp} = \mathrm{Cent}_{{\tiny \Perm(G)}}(N). \]
In \cite{KT20} the Hopf-Galois structure corresponding to $ N^{opp} $ is called the \textit{opposite} of the one corresponding to $ N $, and \cite[Corollary 7.2]{KT22} shows that for each $ g \in G $ we have $ (N_{g})^{opp} = (N^{opp})_{g} $. We shall see fruitful applications of this interaction in Section \ref{sec_intergal_module_structure}. 

\begin{example}
Recall the hypotheses and notation of Example \ref{example_dihedral} above. The opposite of the subgroup $ N $ constructed in that example is the subgroup generated by $ \rho(s) $ and the permutation $ \mu_{0} $ defined by
\[ \mu_{0}[r^{i}s^{j}] = r^{i+(-1)^{i+j}} s^{j}. \]
Exploiting the interaction between opposite subgroups and $ \rho $-conjugation, we find that for each $ k=0, \ldots ,n/2-1 $ the opposite of the subgroup $ N_{r^{k}} $ is generated by $ \rho(s) $ and the permutation $ \mu_{k} $ defined by
\[ \mu_{k}[r^{i}s^{j}] = r^{i+(-1)^{i+j+k}} s^{j}. \]
\end{example}

The Hopf algebra giving the Hopf-Galois structure corresponding to a $ G $-stable regular subgroup $ N $ is the fixed ring $ L[N]^{G} $, where $ G $ acts on $ L $ via the usual Galois action and on $ N $ via $ g \ast \eta = \lambda(g) \eta \lambda(g)^{-1} $. It is possible for two distinct Hopf-Galois structures on $ L/K $ to involve isomorphic Hopf algebras, and it is well known (see \cite[Corollary 2.3]{KKTU19a}, for example) that we have $ L[N]^{G} \cong L[N']^{G} $ if and only if there is a group isomorphism $ \phi : N \xrightarrow{\sim} N' $ such that $ h \ast \phi(\eta) = \phi(h \ast \eta) $ for all $ \eta \in N $ and all $ h \in G $. The proof of Proposition \ref{prop_rho_conjugate_G_stable_regular} shows that the isomorphism $ \phi : N \rightarrow N_{g} $ defined by $ \phi(\eta) = \rho(g)\eta\rho(g)^{-1} $ has this property, and so $ L[N]^{G} \cong L[N_{g}]^{G} $ as $ K $-Hopf algebras (see \cite[Example 2.7]{KKTU19a}). More explicitly: $ \phi $ induces an isomorphism of $ L $-Hopf algebras $ \phi : L[N] \xrightarrow{\sim} L[N_{g}] $, which descends to an isomorphism of $ K $-Hopf algebras $ \phi : L[N]^{G} \xrightarrow{\sim} L[N_{g}]^{G} $, as follows 

\begin{equation} \label{eqn_HA_isomorphism}
\phi \left( \sum_{\eta \in N} c_{\eta} \right) = \sum_{\eta \in N} c_{\eta} \rho(g) \eta \rho(g)^{-1}. 
\end{equation}

\section{The Hopf-Galois correspondence} \label{sec_HG_correspondence}

If a Hopf algebra $ H $ gives a Hopf-Galois structure on a finite extension of fields $ L/K $ then each Hopf subalgebra $ H' $ of $ H $ has a corresponding fixed field
\[ L^{H'} = \{ x \in L \mid z \cdot x = \varepsilon(z)x \mbox{ for all } z \in H' \}, \]
where $ \varepsilon: H \rightarrow K $ denotes the counit map of $ H $. In analogy with the classical Galois correspondence, we have $ [L:L^{H'}] = \dim_{K}(H') $. The correspondence from Hopf subalgebras of $ H $ to intermediate fields of the extension $ L/K $ is inclusion reversing and injective, but not surjective in general; we say that an intermediate field $ M $ is \textit{realizable with respect to $ H $} to mean that there is a Hopf subalgebra $ H' $ of $ H $ such that $ L^{H'} = M $. 

Specializing to the case in which $ L/K $ is a Galois extension, we have seen that each Hopf algebra giving a Hopf-Galois structure on $ L/K $ has the form $ L[N]^{G} $ with $ N $ some $ G $-stable regular subgroup of $ \Perm(G) $; in this case the Hopf subalgebras of $ L[N]^{G} $ are precisely the sets $ L[P]^{G} $ with $ P $ a subgroup of $ N $ that is itself $ G $-stable (see \cite[Theorem 2.3]{CRV16b} or \cite[Theorem 7.2]{HAGMT}). Abusing notation, we write $ L^{P} $ in place of of $ L^{L[P]^{G}} $; since $ \dim_{K}(L[P]^{G}) = |P| $, we then have $ [L:L^{P}]=|P| $. Now let $ g \in G $ and consider the $ G $-stable regular subgroup $ N_{g} $ and the corresponding Hopf-Galois structure, with Hopf algebra $ L[N_{g}]^{G} $. By virtue of the Hopf algebra isomorphism $ \phi : L[N]^{G} \xrightarrow{\sim} L[N_{g}]^{G} $ given in Equation \ref{eqn_HA_isomorphism}, the lattices of Hopf subalgebras of $ L[N]^{G} $ and $ L[N_{g}]^{G} $ are isomorphic; hence the lattices of intermediate fields of $ L/K $ realizable with respect to these Hopf-Galois structures are isomorphic. We can give a more precise statement by studying the actions of $ L[N]^{G} $ and $ L[N_{g}]^{G} $ on $ L $ in more detail. 

It follows from the theorem of Greither and Pareigis that the action of an element $ \sum_{\eta \in N} c_{\eta} \eta \in L[N]^{G} $ on $ L $ is given by 
\begin{equation} \label{eqn_GP_action}
\left( \sum_{\eta \in N} c_{\eta} \eta \right) \cdot x = \sum_{\eta \in N} c_{\eta} \eta^{-1}[e_{G}](x) \mbox{ for all } x \in L 
\end{equation}
(see \cite[Proposition 4.14]{HAGMT}, for example). The action of $ L[N_{g}]^{G} $ on $ L $ is given by an analogous formula. We describe how the isomorphism $ \phi $ interacts with these actions. 

\begin{proposition} \label{prop_conjugation_action}
Let $ x \in L $. Then for all $ z \in L[N]^{G} $ we have
\[ \phi(z) \cdot x = g ( z \cdot g^{-1}(x) ). \]
\end{proposition}
\begin{proof}
Write $ z = \sum_{\eta \in N} c_{\eta} \eta $ with $ c_{\eta} \in L $. Since $ z \in L[N]^{G} $ we have $ z = h \ast z $ for all $ h \in G $. In particular, 
\[ z = g \ast z = \sum_{\eta \in N} g(c_{\eta}) \lambda(g) \eta \lambda(g)^{-1}, \]
and so
\begin{eqnarray*}
\phi(z) & = & \sum_{\eta \in N} g(c_{\eta}) \rho(g) \lambda(g) \eta \lambda(g)^{-1} \rho(g)^{-1} \\
& = & \sum_{\eta \in N} g(c_{\eta})  \lambda(g)\rho(g) \eta \rho(g)^{-1}\lambda(g)^{-1}  \mbox{ by Proposition \ref{prop_rho_conjugate_G_stable_regular}.}
\end{eqnarray*}
Therefore 
\begin{eqnarray*}
\phi(z) \cdot x & = & \left( \sum_{\eta \in N} g(c_{\eta})  \lambda(g)\rho(g) \eta \rho(g)^{-1}\lambda(g)^{-1} \right) \cdot x \\
& = & \sum_{\eta \in N} g(c_{\eta})  \lambda(g)\rho(g) \eta^{-1} \rho(g)^{-1}\lambda(g)^{-1}[e_{G}](x) \\
& = & \sum_{\eta \in N} g(c_{\eta})  \lambda(g)\rho(g) \eta^{-1} [g^{-1}g] (x) \\
& = & \sum_{\eta \in N} g(c_{\eta})  g \eta^{-1} [e_{G}] g^{-1} (x) \\
& = & g \left( \sum_{\eta \in N} c_{\eta} \eta^{-1} [e_{G}] g^{-1} (x) \right) \\
& = & g \left( \sum_{\eta \in N} c_{\eta} \eta \right) \cdot g^{-1} (x)  \\
& = & g \left( z  \cdot g^{-1} (x)\right),
\end{eqnarray*}
as claimed. 
\end{proof}

Now we have the following:

\begin{theorem}
An intermediate field $ M $ of $ L/K $ is realizable with respect to $ L[N]^{G} $ if and only if the intermediate field $ g(M) $ is realizable with respect to $ L[N_{g}]^{G} $. 
\end{theorem}
\begin{proof}
Suppose that $ M $ is realizable with respect to $ L[N]^{G} $. Then $ M = L^{P} $ for some $ G $-stable subgroup $ P $ of $ N $. We have $ \phi(L[P]^{G}) = L[P_{g}]^{G} $, a Hopf subalgebra of $ L[N_{g}]^{G} $, and for $ x \in L $ we have
\begin{eqnarray*}
x \in L^{P_{g}} & \Leftrightarrow & \phi(z) \cdot x = \varepsilon(z)x \mbox{ for all } z \in L[P]^{G} \\
& \Leftrightarrow & g ( z \cdot g^{-1}(x) ) = \varepsilon(z) x \mbox{ for all } z \in L[P]^{G} \mbox{ by Proposition \ref{prop_conjugation_action}} \\
& \Leftrightarrow & z \cdot g^{-1}(x) = \varepsilon(z) g^{-1} (x) \mbox{ for all } z \in L[P]^{G} \\
& \Leftrightarrow & g^{-1}(x) \in L^{P} = M \\
& \Leftrightarrow & x \in g(M). 
\end{eqnarray*}
Hence $ g(M) $ is realizable with respect to $ L[N_{g}]^{G} $. For the converse statement we replace $ N $ by $ N_{g} $ and $ g $ by $ g^{-1} $ in the argument above. 
\end{proof}

\begin{example}
Recall the hypotheses and notation of Example \ref{example_pq} above. Each of the subgroups $ N_{s^{i}} $ is cyclic of order $ pq $, and therefore has a unique subgroup $ P_{i} $ of order $ p $ and a unique subgroup $ Q_{i} $ of order $ q $; by uniqueness, each of these is $ G $-stable. Since $ N_{s^{i}} $ is generated by $ \lambda(s) \rho(s^{i(1-d)}t) $, we see that $ P_{i} = \langle \lambda(s) \rangle $ regardless of $ i $, and that $ Q_{i} = \langle \rho(s^{i(1-d)}t) \rangle $. Therefore $ L^{P_{i}} = L^{\langle s \rangle} $ (the unique intermediate field of degree $ p $ over $ K $) and $ L^{Q} = L^{\langle s^{i(1-d)}t \rangle} $. Thus each of this family of $ \rho $-conjugate Hopf-Galois structures realizes the intermediate fields $ L, K, L^{\langle s \rangle} $, and a unique intermediate field of degree $ q $ over $ K $, and each intermediate field of degree $ q $ over $ K $ is realized by exactly one Hopf-Galois structure in this family. 
\end{example}

\section{Skew braces} \label{sec_skew_braces}

A (left) skew brace is triple $ (B,\star,\circ) $ where $ B $ is a set and $\star,\circ $ are binary operations on $ B $, each making $ B $ into a group, such that
\begin{equation} \label{eqn_brace_relation}
x \circ (y \star z) = (x \circ y) \star x^{-1} \star (x \circ z) \mbox{ for all } x,y,z \in B.
\end{equation}
Here $ x^{-1} $ denotes the inverse of $ x $ with respect to $ \star $; we denote the inverse of $ x $ with respect to $ \circ $ by $ \bar{x} $. We call Equation \ref{eqn_brace_relation} the \textit{(left) skew brace relation}; the right skew brace relation is analogous. We call a skew brace \textit{two-sided} if both the left and right skew brace relations hold for all triples of elements of $ B $. Skew braces with underlying set $ B $ yield solutions of the set theoretic Yang-Baxter equation on $ B $; we can then obtain solutions of the Yang-Baxter equation over any field by using $ B $ as the basis for a vector space (see \cite[\S 2]{KST20}, for example).

The connection between skew braces and Hopf-Galois structures on Galois extensions is noted by Bachiller \cite{Ba16}. One formulation is as follows: a finite skew brace $ (B,\star,\circ) $, yields two left regular representation maps $ \lambda_{\star}, \lambda_{\circ} : B \hookrightarrow \Perm(B) $; we find that $ \lambda_{\star}(B) $ is a $ (B,\circ) $-stable regular subgroup of $ \Perm(B) $, and so corresponds to a Hopf-Galois structure on a Galois extension with Galois group $ (B,\circ) $. Conversely, if $ G = (G,\circ) $ is a finite group and $ N $ is a $ G $-stable regular subgroup of $ \Perm(G) $ then the map $ \eta \mapsto \eta[e_{G}] $ is a bijection from $ N $ to $ G $; we use this to define a second binary operation $ \star $ on $ G $ by
\[ \eta[e_{G}] \star \mu[e_{G}] = \eta\mu[e_{G}]. \]
Then $ (G,\star) \cong N $, and the fact that $ N $ is $ G $-stable implies that $ (G, \star, \circ) $ is a skew brace. As noted in \cite[Proposition 2.1]{NZ19} and \cite[Proposition 3.1]{KT22}, two $ G $-stable regular subgroups $ N, N' \leq \Perm(G) $ yield isomorphic skew braces if and only if $ N' = \varphi^{-1} N \varphi $ for some $ \varphi \in \Aut(G,\circ) $, yield the same skew brace if and only if $ N' = \varphi^{-1} N \varphi $ for some $ \varphi \in \Aut(G, \star, \circ) $ (the group of automorphisms of $ (G,\circ) $ that also respect $ \star $). 

In \cite[Proposition 4.1]{KT22} we find the following alternative formulation of $ \rho $-conjugate subgroups, which enables us to study the corresponding skew braces. 

\begin{proposition} \label{prop_inner_automorphism_formulation}
Let $ N $ be a $ G $-stable regular subgroup of $ \Perm(G) $, and let $ g \in G $. Then $ N_{g} = \varphi_{g} N \varphi_{g}^{-1} $, where $ \varphi_{g} $ is the inner automorphism of $ G $ corresponding to $ g $.
\end{proposition}
\begin{proof}
We may write $ \varphi_{g} = \rho(g) \lambda(g) $, and we then have
\begin{eqnarray*}
\varphi_{g} N \varphi_{g}^{-1} & = & \rho(g) \lambda(g) N \lambda(g)^{-1} \rho(g)^{-1} \\
& = & \rho(g) N \rho(g)^{-1} \mbox{ since $ N $ is $ G $-stable } \\
& = & N_{g}. 
\end{eqnarray*}
\end{proof}

This observation immediately implies that $ \rho $-conjugate Hopf-Galois structures yield isomorphic skew braces under the correspondence described above (\cite[Corollary 4.1]{KT22}). 

The number of distinct $ \rho $-conjugates of a given $ G $-stable regular subgroup $ N $ of $ \Perm(G) $ can studied via the corresponding skew brace.

\begin{theorem} \label{thm_skew_brace_main}
Let $ N $ be a $ G $-stable regular subgroup of $ \Perm(G) $, let $ (G,\star,\circ) $ be the corresponding skew brace, and let $ g \in G $. The following are equivalent:
\begin{enumerate}
\item The $ G $-stable regular subgroup $ N $ is normalized by $ \rho(g) $;
\item The inner automorphism $ \varphi_{g} $ of $ \Aut(G,\circ) $ is an element of $ \Aut(G,\star,\circ) $;
\item We have $ (y \star z) \circ g = (y \circ g ) \star g^{-1} \star (z \circ g) $ for all $ y,z \in G $. 
\end{enumerate}
\end{theorem}

\begin{corollary}
Let $ G' = \{ g \in G \mid \varphi_{g} \in \Aut(G,\star,\circ) \} $. Then the $ G $-stable regular subgroup $ N $ has $ |G|/|G'| $ distinct $ \rho $-conjugates.
\end{corollary}

\begin{corollary}
The $ G $-stable regular subgroup $ N $ is normalized by $ \rho(G) $ if and only if $ (G,\star,\circ) $ is a two-sided skew brace. 
\end{corollary}

\begin{proof}[Proof of Theorem \ref{thm_skew_brace_main}.]
By Proposition \ref{prop_inner_automorphism_formulation} we have $ N_{g} = \varphi_{g}N\varphi_{g}^{-1} $, and the discussion above then implies that $ N_{g} = N $ if and only if $ \varphi_{g} \in \Aut(G,\star,\circ) $. Hence (i) and (ii) are equivalent. 

Now suppose that (ii) holds. By the left skew brace relation \eqref{eqn_brace_relation}, for all $ y,z \in G $ we have 
\[ g \circ (y \star z) = (g \circ y) \star g^{-1} \star (g \circ z). \]
Since $ \varphi_{g} \in \Aut(G,\star,\circ) $, so is $ \varphi_{g}^{-1} = \varphi_{\bar{g}} $. Applying this to the equation above yields
\begin{eqnarray*}
&& \varphi_{\bar{g}}(g \circ (y \star z)) = \varphi_{\bar{g}}((g \circ y) \star g^{-1} \star (g \circ z)) \\
& \Rightarrow & \varphi_{\bar{g}}(g \circ (y \star z)) = \varphi_{\bar{g}}(g \circ y) \star \varphi_{\bar{g}}(g^{-1}) \star \varphi_{\bar{g}}(g \circ z) \\
& \Rightarrow & \bar{g} \circ g \circ (y \star z) \circ g = (\bar{g} \circ g \circ y \circ g) \star \varphi_{\bar{g}}(g)^{-1} \star (\bar{g} \circ g \circ z \circ g) \\
& \Rightarrow & (y \star z) \circ g = (y \circ g) \star g^{-1} \star (z \circ g);
\end{eqnarray*}
Hence (iii) holds. 

Finally, suppose that (iii) holds. Then for all $ y,z \in G $ we have
\begin{eqnarray*}
\varphi_{\bar{g}}(y \star z) & = & \bar{g} \circ (y \star z) \circ g \\
& = & \bar{g} \circ ( (y \circ g) \star g^{-1} \star (z \circ g)) \mbox{ by (iii) } \\
& = & (\bar{g} \circ y \circ g) \star \bar{g}^{-1} \star (\bar{g} \circ g^{-1}) \star \bar{g}^{-1} \star (\bar{g} \circ z \circ g) \mbox{ by Equation \ref{eqn_brace_relation}. }
\end{eqnarray*} 
Now by a well known skew brace identity (see \cite[Lemma 1.7, part (2)]{GV17}, for example) we have 
\[ \bar{g}^{-1} \star (\bar{g} \circ g^{-1}) \star \bar{g}^{-1} = (\bar{g} \circ g)^{-1} = e_{G}; \]
applying this to the final line above yields
\[ \varphi_{\bar{g}}(y \star z) = (\bar{g} \circ y \circ g) \star (\bar{g} \circ z \circ g) = \varphi_{\bar{g}}(y) \star \varphi_{\bar{g}}(z), \]
and so $ \varphi_{\bar{g}} \in \Aut(G,\star,\circ) $. Therefore $ \varphi_{g} \in \Aut(G,\star,\circ)) $, so (ii) holds. 
\end{proof}

\section{Integral module structure} \label{sec_intergal_module_structure}

We now suppose that $ L/K $ is a Galois extension of local or global fields (in any characteristic) with Galois group $ G $. As discussed  briefly in Section \ref{sec_introduction}, Hopf-Galois structures can be used in this context to study the structure of fractional ideals $ \B $ of $ L $. More precisely: if $ H $ is a Hopf algebra giving a Hopf-Galois structure on the extension then $ L $ is a free $ H $-module of rank $ 1 $ (this is a Hopf-Galois analogue of the classical normal basis theorem), and we seek integral analogues of this result. Each fractional ideal $ \B $ of $ L $ has an \textit{associated order} in $ H $:
\[ \A_{H}(\B) = \{ h \in H \mid h \cdot x \in \B \mbox{ for all } x \in \B \}, \]
and we may study the structure of $ \B $ as an $ \A_{H}(\B) $-module, with particular emphasis on determining criteria for it to be free (necessarily of rank $ 1 $). By construction, $ \A_{H}(\B) $ is the largest subring of $ H $ for which $ \B $ is a module, and it is the only order in $ H $ over which $ \B $ can possibly be free.  

The most famous results in this area are due to Byott \cite{By97}, who exhibits Galois extensions $ L/K $ of $ p $-adic fields for which the valuation ring $ \OL $ is not free over $ \A_{K[G]}(\OL) $, but is free over $ \A_{H}(\OL) $ for some other Hopf algebra $ H $ giving a Hopf-Galois structure on the extension. On the other hand, in \cite{Tr16a} and \cite{Tr18b} we show that a fractional ideal $ \B $ of $ L $ is free over its associated order in a particular Hopf-Galois structure if and only if it is free over its associated order in the opposite Hopf-Galois structure. The main result of this section is in a similar vein. Recall that an \textit{ambiguous ideal} of $ L $ is a fractional idea of $ L $ that is invariant under the action of $ G $. (In particular: the ring of algebraic integers $ \OL $ is an ambiguous ideal of $ L $, and if $ L/K $ is an extension of local fields then every fractional ideal of $ L $ is an ambiguous ideal.) We will prove:

\begin{theorem} \label{thm_integral_main}
Let $ L[N]^{G} $ give a Hopf-Galois structure on $ L/K $, and let $ g \in G $. Then an ambiguous ideal $ \B $ of $ L $ is free over its associated order in $ L[N]^{G} $ if and only if it is free over its associated order in $ L[N_{g}]^{G} $.  
\end{theorem}

\begin{corollary}
An ambiguous ideal of $ L $ is free over its associated order in $ L[N]^{G} $ if and only if it is free over its associated order in $ L[N_{g}]^{G} $ for all $ g \in G $. 
\end{corollary}

The proof of Theorem \ref{thm_integral_main} does not depend upon the fact that $ L $ is a field, and so we have

\begin{corollary}
If $ L/K $ is an extension of global fields then an ambiguous ideal of $ L $ is locally free over its associated order in $ L[N]^{G} $ if and only if it is locally free over its associated order in $ L[N_{g}]^{G} $ for all $ g \in G $. 
\end{corollary}

\begin{proof}[Proof of Theorem \ref{thm_integral_main}.]
Let $ \A $ denote the associated order of $ \B $ in $ L[N]^{G} $, and let $ \A_{g}  $ denote the associated order of $ \B $ in $ L[N_{g}]^{G} $. Suppose that $ \B $ is a free $ \A $-module, with generator $ x \in \B $. We shall show that $ \A_{g} = \phi(\A) $ and that $ \B $ is a free $ \A_{g} $-module with generator $ g(x) $ (note that $ g(x) \in \B $ because $ \B $ is an ambiguous ideal of $ L $).  

Since $ \phi $ is an isomorphism of $ K $-Hopf algebras, it is immediate that $ \phi(\A) $ is an order in $ L[N_{g}]^{G} $. Now let $ y \in \B $. Since $ \B $ is an ambiguous ideal of $ L $ we have $ g^{-1}(y) \in \B $, and since $ \B $ is a free $ \A $-module with generator $ x $ there exists a unique element $ z \in \A $ such that $ z \cdot x = g^{-1}(y) $. Now consider the action of the element $ \phi(z) \in \phi(\A) $ on the element $ g(x) \in \B $. We have:
\begin{eqnarray*}
\phi(z) \cdot g(x) & = & g ( z \cdot g^{-1}( g(x)) ) \mbox{ by Proposition \ref{prop_conjugation_action}) }\\
& = & g ( g^{-1}(y)) \\
& = & y.
\end{eqnarray*}
Hence given $ y \in \B $ there exists a unique element $ \phi(z) $ of the order $ \phi(\A) $ in $ L[N_{g}]^{G} $ such that $ \phi(z) \cdot g(x) = y $, and so $ \B $ is a free $ \phi(\A) $-module with generator $ g(x) $. Since $ \A_{g} $ is the only order in $ L[N_{g}]^{G} $ over which $ \B $ can be free, we conclude that $ \phi(\A) = \A_{g} $. 

For the converse statement we replace $ N $ by $ N_{g} $ and $ g $ by $ g^{-1} $ in the argument above. 
\end{proof}

The interaction between Theorem \ref{thm_integral_main} and the previously mentioned results concerning integral module structure with respect to opposite Hopf-Galois structures can cause one instance of freeness to propagate to a whole family.

\begin{example}
Let $ p,q $ be odd primes with $ p \equiv 1 \pmod{q} $ and let $ L/K $ be a Galois extension of local or global fields whose Galois group $ G $ is isomorphic to the metacyclic group of order $ pq $
\[ G = \langle s, t \mid s^{p}=t^{q}=e, \; t s t^{-1} = s^{d} \rangle, \]
as in Example \ref{example_pq}. For $ k=0, \ldots , p-1 $ let 
\[ N_{k} = \langle \lambda(s), \lambda(t) \rho(s^{k}t) \rangle. \]
It is routine to verify that each $ N_{k} $ is distinct, $ G $-stable, regular, and isomorphic to $ G $. We find that $ \rho(t) $ normalizes each $ N_{k} $, and that $ \rho(s)N_{k}\rho(s^{-1}) = N_{k+1-d} $. Thus the subgroups $ N_{k} $ are mutually $ \rho $-conjugate, and by \cite[Corollary 7.2]{KT22} the same is true for their opposites $ N_{k}^{opp} $. 

Now we write $ H_{k} $ (repectively $ H_{k}^{opp} $) for the Hopf algebra giving the Hopf-Galois structure corresponding to $ N_{k} $ (respectively $ H_{k}^{opp} $), and consider questions of integral module structure. By Theorem \ref{thm_integral_main}, if the ring of integers $ \OL $ is free over its associated order in one of the $ H_{k} $ then it is free over its associated order in each of the $ H_{k} $, and by \cite[Theorem 1.2]{Tr18b} if it is free over its associated order in a particular $ H_{k} $ then it is free over its associated order in $ H_{k}^{opp} $. Thus $ \OL $ is either free over its associated orders in all of these $ 2p $ Hopf-Galois structures or none of them. 
\end{example} 

\section{Interactions with existing constructions} \label{sec_interactions}

In this section we study the interactions between $ \rho $-conjugation and various results concerning the construction and enumeration of Hopf-Galois structures on a Galois extension $ L/K $. (We no longer assume that $ L/K $ is an extension of local or global fields.)

\subsection{Byott's translation theorem and fixed-point free pairs of homomorphisms}

Many results concerning the construction and enumeration of Hopf-Galois structures on separable field extensions employ Byott's translation theorem \cite{By96} to address some of the combinatorial difficulties inherent to Greither-Pareigis theory. Specializing to the Galois case, Byott's theorem may be summarized as follows: let $ M $ be an abstract group of the same order as $ G $, and consider {\em regular embeddings} $ \alpha : M \rightarrow \Perm(G) $ (that is: embeddings with regular image). Such an embedding yields a bijection $ a : M \rightarrow G $ defined by $ a(\mu) = \alpha(\mu)[e_{G}] $ for all $ \mu \in M $, and thence a regular embedding $ \beta : G \rightarrow \Perm(M) $ defined by $ \beta(g)[\mu] = a^{-1} \lambda(g) a $ for all $ g \in G $. Byott's theorem asserts that this construction yields a bijection between regular embeddings of $ M $ into $ \Perm(G) $ and regular embeddings of $ G $ into $ \Perm(M) $. Furthermore: $ \alpha(M) $ is $ G $-stable if and only if $ \beta(G) \subseteq \Hol(M) = \lambda(M) \Aut(M) \subseteq \Perm(M) $, and $ \alpha(M)=\alpha'(M) $ if and only if there exists $ \theta \in \Aut(M) $ such that $ \beta'(g) = \theta \beta(g) \theta^{-1} $ for all $ g \in G $. 

We can study $ \rho $-conjugate Hopf-Galois structures from this perspective, using the reformulation expressed in Proposition \ref{prop_inner_automorphism_formulation}. Given a regular embedding $ \alpha : M \rightarrow \Perm(G) $ and an element $ g \in G $, the map $ \alpha_{g} : M \rightarrow \Perm(G) $ defined by $ \alpha_{g}(\mu) = \varphi_{g} \alpha(\mu) \varphi_{g}^{-1} $ for all $ \mu \in M $ is a regular embedding with image $ \alpha(M)_{g} $. We then find the following:

\begin{lemma} \label{lemma_beta_conjugation}
Let $ \beta : G \rightarrow \Hol(M) $ be the embedding corresponding to $ \alpha $. Then then embedding $ \beta_{g} : G \rightarrow \Hol(M) $ corresponding to $ \alpha_{g} $ is given by $ \beta_{g} = \beta \varphi_{g^{-1}} $. 
\end{lemma}
\begin{proof}
The bijection $ a_{g} : M \rightarrow G $ corresponding to $ \alpha_{g} $ is given by 
\[ a_{g}(\mu) =  \varphi_{g} \alpha(\mu) \varphi_{g^{-1}}[e_{G}] = \varphi_{g} \alpha(\mu)[e_{G}] = \varphi_{g} a(\mu). \]
Hence for $ h \in G $ and $ \mu \in M $ we have
\begin{eqnarray*}
\beta_{g}(h)[\mu] & = & a^{-1} \varphi_{g}^{-1} \lambda_{G}(h) \barphi_{g} a [\mu] \\
& = & a^{-1} \varphi_{g}^{-1} (h g a [\mu]  g^{-1}) \\
& = & a^{-1}( g^{-1} h g a [\mu]  g^{-1} g )\\
& = & a^{-1} \lambda(\varphi_{g^{-1}}(h)) a [\mu] \\
& = & \beta(\varphi_{g^{-1}}(h))[\mu].
\end{eqnarray*}
Hence $ \beta_{g} = \beta \varphi_{g^{-1}} $, as claimed. 
\end{proof}

As an application of Lemma \ref{lemma_beta_conjugation} we study the interaction between $ \rho $-conjugation and Hopf-Galois structures on Galois extensions arising from \textit{fixed point free} pairs of homomorphisms, due to Byott and Childs \cite{BC12}.  Given finite groups $ G,M $ of the same order, a pair of homomorphisms $ f_{1}, f_{2} : G \rightarrow M $ is called fixed point free if the only element $ h \in G $ for which $ f_{1}(h)=f_{2}(h) $ is $ h = e_{G} $. Given such a pair of homomorphisms, the map $ \beta_{f_{1},f_{2}} : G \rightarrow \Hol(M) $ defined by
\[ \beta_{f_{1},f_{2}}(h) = \lambda(f_{1}(h)) \rho(f_{2}(h)) \mbox{ for all } h \in G \]
is a regular embedding of $ G $ into $ \Hol(M) $ (where $ \lambda, \rho $ denote the left and right regular representations of $ M $), and therefore yields a Hopf-Galois structure of type $ M $ on a Galois extension with Galois group $ G $. Applying Lemma \ref{lemma_beta_conjugation} we obtain immediately

\begin{proposition}
Let $ f_{1},f_{2} : G \rightarrow \Hol(M) $ be a fixed point free pair of homomorphisms, with corresponding regular embedding $ \beta : G \rightarrow \Hol(M) $, and let $ g \in G $. For $ i=1,2 $ let $ f_{i,g} = f_{i}  \varphi_{g} $. Then $ f_{1,g}, f_{2,g} $ is a fixed point free pair of homomorphisms, and the corresponding embedding of $ G $ into $ \Hol(M) $ is $ \beta_{g} $. 
\end{proposition}

\begin{example}
Let $ n $ be an even natural number and let $ G = \langle r,s \rangle $ and $ M = \langle \mu, \pi \rangle $ both be isomorphic to the dihedral group of order $ 2n $ (here $ r, \mu $ have order $ n $ and $ s, \pi $ have order $ 2 $). Fix $ 0 \leq k \leq n/2-1 $ and define $ f_{1} : G \rightarrow M $ by $ f_{1}(r^{i}s^{j}) = \mu^{i}\pi^{j} $ and $ f_{2}: G \rightarrow M $ by $ f_{2}(r^{i}s^{j}) = (\mu^{2k}\pi)^{i} $. Then $ f_{1}, f_{2} $ form a fixed point free pair of homomorphisms, and the corresponding embedding $ \beta : G \hookrightarrow \Hol(M) $ is given by 
\[ \beta(r) = \lambda(\mu) \rho(\mu^{2k}\pi), \quad \beta(s) = \lambda(\pi). \]
(Here $ \lambda, \rho $ denote the left and right regular representations of $ M $.) We find that composing $ \beta $ with $ \varphi_{s} $ results in an equivalent embedding, whereas composing with each $ \varphi_{r^{\ell}} $ for $ 0 \leq \ell \leq n/2-1 $ gives an inequivalent embedding. Thus we obtain $ n/2 $ mutually $ \rho $-conjugate Hopf-Galois structures of dihedral type on a Galois extension with Galois group isomorphic to $ G $. In fact, these are the Hopf-Galois structures described in Example \ref{example_dihedral}.
\end{example}

\subsection{Hopf-Galois structures arising from abelian maps}

An \textit{abelian map} is a group endomorphism with abelian image. In \cite{Ch13} Childs shows that fixed point free abelian maps on a finite group $ G $ yield a family of Hopf-Galois structures on a Galois extension $ L/K $ with Galois group $ G $. This approach is generalized by Koch \cite{Koc21a}, as follows: let $ \psi : G \rightarrow G $ be an abelian map, and for each $ h \in G $ define $ \eta_{\psi}(h) \in \Perm(G) $ by 
\[ \eta_{\psi}(h) = \lambda(h \psi(h)^{-1}) \rho(\psi(h)^{-1}). \]
Then $ N_{\psi} = \{ \eta_{\psi}(h) \mid h \in G \} $ is a $ G $-stable regular subgroup of $ \Perm(G) $ \cite[Theorem 3.1]{Koc21a}, which therefore corresponds to a Hopf-Galois structure on $ L/K $. 

\begin{proposition} \label{prop_conjugation_abelian_maps}
Let $ \psi $ be an abelian map on $ G $ and let $ N_{\psi} $ be the corresponding $ G $-stable regular subgroup of $ \Perm(G) $. Then for each $ g \in G $ the subgroup $ N_{\psi,g} $ arises from an abelian map on $ G $. 
\end{proposition}
\begin{proof}
By Proposition \ref{prop_inner_automorphism_formulation} we have $ N_{\psi,g} = \varphi_{g} N_{\psi} \varphi_{g}^{-1} $. It is shown in \cite[Proposition 5.1]{KT22} that if $ \psi $ is an abelian map on $ G $ and $ \varphi \in \Aut(G) $ then the map $ \varphi \psi \varphi^{-1} $ is also an abelian map on $ G $, and $ N_{\varphi \psi \varphi^{-1}} = \varphi N_{\psi} \varphi^{-1} $. Thus $ N_{\psi,g} $ arises from the abelian map $ \psi' = \varphi_{g} \psi \varphi_{g}^{-1} $ on $ G $. 
\end{proof}

\begin{example}
Let $ n \geq 5 $ and suppose that $ G $ is isomorphic to the symmetric group $ S_{n} $. In \cite[Example 3.7]{Koc21a} Koch shows that the abelian maps on $ G $ correspond bijectively with elements $ x \in S_{n} $ of order at most $ 2 $, as follows:
\[ \psi_{x}(t) = \left\{ \begin{array}{ll} e & \mbox{if } t \in A_{n} \\ x & \mbox{if } t \not \in A_{n} \end{array} \right. \]
Let $ N_{x} $ be the $ G $-stable regular subgroup arising from the abelian map $ \psi_{x} $. Applying Proposition \ref{prop_conjugation_abelian_maps} we see that for each $ g \in G $ the subgroup $ N_{\psi,g} $ arises from the abelian map $ \varphi_{g} \psi_{x} \varphi_{g}^{-1} $, which behaves as follows
\[ \varphi_{g} \psi_{x} \varphi_{g}^{-1}(t) = \left\{ \begin{array}{ll} e & \mbox{if } t \in A_{n} \\ gxg^{-1} & \mbox{if } t \not \in A_{n} \end{array} \right. \]
Thus $ \varphi_{g} \psi_{x} \varphi_{g}^{-1} = \psi_{g x g^{-1}} $, and so the $ \rho $-conjugate subgroups correspond to conjugacy classes of elements of order at most $ 2 $ in $ G $. In particular, the subgroups arising from transpositions are mutually $ \rho $-conjugate. 
\end{example}

\subsection{Induced Hopf-Galois structures}

The method of \textit{induced} Hopf-Galois structures is due to Crespo, Rio, and Vela \cite{CRV16}: if $ T $ is a subgroup of $ G $ having a normal complement $ S $ in $ G $ then, given a Hopf-Galois structure on $ L/L^{T} $ and a Hopf-Galois structure on $ L^{T}/K $, we can induce a Hopf-Galois structure on $ L/K $. Note that $ T $ need not be a normal subgroup of $ G $, and so the extension $ L^{T}/K $ need not be a Galois extension. In order to study Hopf-Galois structures on this extension, we need the full strength of the Greither-Pareigis classification. In this context, it implies that the Hopf-Galois structures on $ L^{T}/K $ correspond bijectively with regular subgroups of $ \Perm(G/T) $ (the permutation group of the left coset space) that are $ G $-stable in the sense that they are normalized by the image of $ G $ under the left translation map $ \lambda : G \rightarrow \Perm(G/T) $ defined by $ \lambda(h)[xT] = hxT $ for all $ h \in G $ and $ xT \in G/T $.

Given a $ G $-stable regular subgroup $ A \subset \Perm(G/T) $ and a $ T $-stable regular subgroup $ B \subset \Perm(T) $, we use the fact that $ G $ is the semidirect product of $ S $ and $ T $ to identify $ A $ with a regular subgroup of $ \Perm(S) $, and then to define a map $ \eta : A \times B \rightarrow \Perm(G) $ by 
\[ \eta(a,b)[s t] = a[s] b[t] \mbox{ for all } s \in S \mbox{ and } t \in T. \] 
It is easy to see that $ \eta $ is an embedding whose image $ N $ is a regular subgroup of $ \Perm(G) $ isomorphic to $ A \times B $; \cite[Theorem 3]{CRV16} shows that $ N $ is $ G $-stable, and hence corresponds to a Hopf-Galois structure on $ L/K $. We say that the subgroup $ N $ is \textit{induced} from the subgroups $ A $ and $ B $, and that the Hopf-Galois structure on $ L/K $ corresponding to $ N $ is induced from those corresponding to $ A $ on $ L^{T}/K $ and $ B $ on $ L/L^{T} $. 

Now let $ g \in G $ and let $ \varphi = \varphi_{g} $ be the inner automorphism corresponding to $ g $. Since $ S $ is normal in $ G $ we have $ \varphi(S) = S $, and it is easily shown that $ G $ is the semidirect product of $ S $ and $ \varphi(T) $. We shall show that the Hopf-Galois structure corresponding to $ N_{g} $ is induced from certain Hopf-Galois structures on $ L^{\varphi(T)}/K $ and $ L/L^{\varphi(T)} $. 

\begin{lemma} \label{lem_induced}
Let $ g \in G $ and $ \varphi = \varphi_{g} $. Then:
\begin{enumerate}
\item $ \varphi A \varphi^{-1} $ identifies naturally with a $ G $-stable regular subgroup of $ \Perm(G / \varphi(T)) $; 
\item $ \varphi B \varphi^{-1} $ identifies naturally with a $ \varphi(T) $-stable regular subgroup of $ \Perm(\varphi(T)) $. 
\end{enumerate}
\end{lemma}
\begin{proof}
To show (i), let $ a \in A $ and $ s\varphi(T) \in G / \varphi(T) $ (with $ s \in S $). Then 
\[ \varphi a \varphi^{-1} [s\varphi(T)] = \varphi a ( \varphi^{-1}(s) T ) = s' \varphi(T) \]
for some element $ s' \in S $. Hence $ \varphi A \varphi^{-1} $ permutes $ G / \varphi(T) $, and this action is regular because $ A $ acts regularly on $ G/T $. To show $ G $-stability, let $ h \in G $ and consider
\begin{eqnarray*}
\lambda(h) \varphi a \varphi^{-1}  \lambda(h^{-1}) [s \varphi(T)] & = & h \varphi a [ \varphi^{-1}(h^{-1}) \varphi^{-1}(s) T ] \\
& = & \varphi( \varphi^{-1}(h) a [ \varphi^{-1}(h^{-1}) \varphi^{-1}(s) T ]  \\
& = & \varphi( \varphi^{-1}(h) a \varphi^{-1}(h^{-1}) \varphi^{-1} [s \varphi(T) ]  \\
& = & \varphi a' \varphi^{-1} [ s \varphi(T) ]
\end{eqnarray*}
for some $ a' \in A $, since $ A $ is $ G $-stable. Hence $ \varphi A \varphi^{-1} $ is $ G $-stable. The proof of (ii) is similar. 
\end{proof}

Note that if $ g \in T $ then the regular subgroup $ \varphi B \varphi^{-1} $ of $ \Perm(T) $ is $ \rho $-conjugate to $ B $. 

\begin{proposition} \label{prop_conjugation_induced}
For $ A, B, N, $ and $ g $ as above, the subgroup $ N_{g} \leq \Perm(G) $ is induced from $ \varphi A \varphi^{-1} \leq \Perm(G/\varphi(T)) $ and $ \varphi B \varphi^{-1}  \leq \Perm(\varphi(T)) $.
\end{proposition}
\begin{proof}
Let $ h \in G $, and write $ h = s\varphi(t) $ with $ s \in S $ and $ t \in T $. Now let $ \eta = \eta(a,b) \in N $ and consider
\begin{eqnarray*}
\varphi \eta \varphi^{-1} [s\varphi(t)] & = & \varphi \eta [\varphi^{-1}(s) t]  \\
& = & \varphi (a[\varphi^{-1}(s)] b[t]) \mbox{ (note that $ \varphi^{-1}(s) \in S $)} \\
& = & (\varphi a[\varphi^{-1}(s)]) (\varphi b[t]) \\
& = & (\varphi a[\varphi^{-1}(s)]) (\varphi b[\varphi^{-1}( \varphi(t))]) \\
& = & \eta( \varphi a \varphi^{-1}, \varphi b \varphi^{-1}) [s \varphi(t)]. 
\end{eqnarray*}
Hence $ N_{g} \leq \Perm(G) $ is induced from $ \varphi A \varphi^{-1} \leq \Perm(G/\varphi(T)) $ and $ \varphi B \varphi^{-1}  \leq \Perm(\varphi(T)) $.
\end{proof}

We remark that the argument of Proposition \ref{prop_conjugation_induced} remains valid if $ \varphi $ is any automorphism of $ G $ that preserves $ S $.  

\begin{example}
Let $ p,q $ be odd primes with $ p \equiv 1 \pmod{q} $ and let $ L/K $ be a Galois extension of fields whose Galois group $ G $ is isomorphic to the metacyclic group of order $ pq $
\[ G = \langle s, t \mid s^{p}=t^{q}=e, \; t s t^{-1} = s^{d} \rangle, \]
as in Example \ref{example_pq}. Let $ T = \langle t \rangle $ and $ S = \langle s \rangle $. Then the conjugates of $ T $ are precisely the subgroups $ T_{i} = s^{i} T s^{-i} $, and $ S $ is a normal complement to each of these subgroups in $ G $. For each $ i $ the extensions $ L/L^{T_{i}} $ and $ L^{T_{i}}/K $ have prime degree and each admits a unique Hopf-Galois structure (note that each $ L^{T_{i}}/K $ is non-normal). In each case, we can induce a Hopf-Galois structure on $ L/K $ from these, and by Proposition \ref{prop_conjugation_induced} we obtain a family of $ p $ mutually $ \rho $-conjugate Hopf-Galois on $ L/K $. Referring to Example \ref{example_pq} we see that these are all of the Hopf-Galois structures of cyclic type admitted by $ L/K $. 
\end{example}


\begin{thebibliography}{10}

\bibitem{Ba16}
D.~Bachiller.
\newblock Counterexample to a conjecture about braces.
\newblock {\em J. Algebra}, 453:160--176, 2016.

\bibitem{By96}
N.~P. Byott.
\newblock Uniqueness of {H}opf {G}alois structure of separable field
  extensions.
\newblock {\em Comm. Algebra}, 24:3217--3228, 3705, 1996.

\bibitem{By97}
N.~P. Byott.
\newblock Galois structure of ideals in wildly ramified abelian $p$-extensions
  of a $p$-adic field, and some applications.
\newblock {\em J. Th\'{e}or.\ Nombres Bordeaux}, 9:201--219, 1997.

\bibitem{By04c}
N.~P. Byott.
\newblock {H}opf-{G}alois structures on {G}alois field extensions of degree
  $pq$.
\newblock {\em J. Pure Appl. Algebra}, 188(1-3,2.2):45--57, 2004.

\bibitem{BC12}
N.~P. Byott and L.~N. Childs.
\newblock Fixed point free pairs of homomorphisms and {H}opf {G}alois
  structures.
\newblock {\em New York J. Math.}, 18:707--731, 2012.

\bibitem{TWE}
L.~N. Childs.
\newblock {\em Taming Wild Extensions: {H}opf algebras and local {G}alois
  module theory}, volume~80 of {\em Mathematical Surveys and Monographs}.
\newblock {A}merican Mathematical Society, 2000.

\bibitem{Ch13}
L.~N. Childs.
\newblock Fixed-point free endomorphisms and {H}opf {G}alois structures.
\newblock {\em Proc. Amer. Math. Soc.}, 141:1255--1265, 2013.

\bibitem{HAGMT}
L.~N. Childs, C.~Greither, K.~P. Keating, A.~Koch, T.~Kohl, P.~J. Truman, and
  R.~Underwood.
\newblock {\em {H}opf algebras and {G}alois module theory}, volume 260 of {\em
  Mathematical Surveys and Monographs}.
\newblock {A}merican Mathematical Society, 2021.

\bibitem{CRV16b}
T.~Crespo, A.~Rio, and M.~Vela.
\newblock The {H}opf {G}alois property in subfield lattices.
\newblock {\em Comm. Algebra}, 44:336--353, 2016.

\bibitem{CRV16}
T.~Crespo, A.~Rio, and M.~Vela.
\newblock On the {G}alois correspondence theorem in separable {H}opf {G}alois
  theory.
\newblock {\em Publ. Mat. (Barcelona)}, 60(1):221--234, 2016.

\bibitem{GP87}
C.~Greither and B.~Pareigis.
\newblock {H}opf {G}alois theory for separable field extensions.
\newblock {\em J. Algebra}, 106:239--258, 1987.

\bibitem{GV17}
L.~Guarneri and L.~Vendramin.
\newblock Skew braces and the {Y}ang-{B}axter equation.
\newblock {\em Math. Comp.}, 86(307):2519--2534, 2017.

\bibitem{Koc21a}
A.~Koch.
\newblock Abelian maps, bi-skew braces, and opposite pairs of {H}opf-{G}alois
  structures.
\newblock {\em Proc. Amer. Math. Soc.}, 8(16):189--203, 2021.

\bibitem{KKTU19a}
A.~Koch, T.~Kohl, P.~J. Truman, and R.~Underwood.
\newblock Isomorphism problems for {H}opf-{G}alois structures on separable
  field extensions.
\newblock {\em J. Pure Appl. Algebra}, 223:2230--2245, 2019.

\bibitem{KST20}
A.~Koch, L.~Stordy, and P.~J. Truman.
\newblock Abelian fixed point free endomorphisms and the yang-baxter equation.
\newblock {\em New York J. Math.}, 26:1473--1492, 2020.

\bibitem{KT20}
A.~Koch and P.~J. Truman.
\newblock Opposite skew left braces and applications.
\newblock {\em J. Algebra}, 546:218--235, 2020.

\bibitem{KT22}
A.~Koch and P.~J. Truman.
\newblock Skew left braces and isomorphism problems for {H}opf-{G}alois
  structures on {G}alois extensions.
\newblock {\em J. Algebra Appl.}, 2022.

\bibitem{NZ19}
K.~{Nejabati Zenouz}.
\newblock Skew braces and {H}opf-{G}alois structures of {H}eisenberg type.
\newblock {\em J. Algebra}, 524:187--225, 2019.

\bibitem{Tr16a}
P.~J. Truman.
\newblock Canonical nonclassical {H}opf--{G}alois module structure of
  nonabelian {G}alois extensions.
\newblock {\em Comm. Algebra}, 44(3):1119--1130, 2016.

\bibitem{Tr18b}
P.~J. Truman.
\newblock Commuting {H}opf-{G}alois structures on a separable extension.
\newblock {\em Comm. Algebra}, 46(4):1420--1427, 2018.

\end{thebibliography}

\end{document}